\documentclass[11pt,twoside,a4paper]{article}
\usepackage{color,amscd,amsmath,amssymb,amsthm,latexsym,stmaryrd,authblk,mathabx,shuffle,hyperref,tikz,tkz-tab,skak,tikz-cd} 
\usepackage[latin1]{inputenc}
\usepackage[T1]{fontenc}   
\usepackage[english]{babel}
\usepackage{eucal}

\newcommand{\Decbot}[1]{\operatorname{Dec}_\bot{}\kern-2pt{#1}}
\newcommand{\CC}{\mathcal{C}}
\newcommand{\isopil}{\stackrel{\raisebox{0.1ex}[0ex][0ex]{\(\sim\)}}%
			{\raisebox{-0.15ex}[0.28ex]{\(\rightarrow\)}}}
\providecommand{\kat}[1]{\text{\textbf{\textsl{#1}}}}

\newcommand{\into}{\rightarrowtail}
\newcommand{\onto}{\twoheadrightarrow}

\DeclareMathAlphabet{\mathbbe}{U}{bbold}{m}{n}
\newcommand{\simplexcategory}{\mathbf{\Delta}}
\newcommand{\op}{^{\text{{\rm{op}}}}}

\renewcommand{\divides}{\mathbin{|}}

\newcommand{\divposet}{(\N^\ast\!, \,|\,)}
\newcommand{\divposetnoparens}{\N^\ast\!, \,|\,}
\newcommand{\multmonoid}{(\N^\ast\!, \,{\cdot}\,)}
\newcommand{\multmonoidnoparens}{\N^\ast\!, \,{\cdot}\,}

\newcommand{\incCoalg}[1]{\mathsf{IncCoalg}(#1)}
\newcommand{\incAlg}[1]{\mathsf{IncAlg}(#1)}
\newcommand{\redIncAlg}[1]{\mathsf{IncAlg}^{\operatorname{red}}(#1)}
\newcommand{\redIncCoalg}[1]{\mathsf{IncCoalg}^{\operatorname{red}}(#1)}

\providecommand{\norm}[1]{\left| {#1}\right|}
\newcommand{\st}{\operatorname{st}}

\tikzset{
  /tikz/commutative diagrams/on top/.style={inner sep=1pt, description}
}

\input{xy}
\xyoption{all}

\theoremstyle{plain}
\newtheorem{theo}{Theorem}[subsection]
\newtheorem{lemma}[theo]{Lemma}
\newtheorem{cor}[theo]{Corollary}
\newtheorem{prop}[theo]{Proposition}

\theoremstyle{definition}
\newtheorem{defi}[theo]{Definition}
\newtheorem{remark}[theo]{Remark}
\newtheorem{notation}[theo]{Notation}

\newcommand{\parti}{\mathcal{NCP}}

  \newcommand{\partipres}[1]{\parti_{{#1}\text{\rm -pres}}{}}

\newcommand{\N}{\mathbb{N}}
\newcommand{\Q}{\mathbb{Q}}
\newcommand{\C}{\mathbb{C}}

\title{Noncrossing arithmetic}

\author{Kurusch Ebrahimi-Fard$^a$, Lo\"\i c Foissy$^b$, Joachim Kock$^c$, Fr\'ed\'eric Patras$^d$\\
{\small \it $^a$Department of Mathematical Sciences}\\
{\small \it Norwegian University of Science and Technology (NTNU)}\\
{\small \it 7491 Trondheim, Norway.}\\
{\small \it and Lie--St{\o}rmer Center  (Troms{\o})}\\
{\small \it email: {\rm kurusch.ebrahimi-fard@ntnu.no}}

\smallskip

{\small \it $^b$Univ. Littoral C\^ote d'Opale, UR 2597}\\
{\small \it Laboratoire de Math\'ematiques Pures et Appliqu\'ees Joseph Liouville}\\
{\small \it F-62100 Calais, France}\\
{\small \it email: {\rm foissy@univ-littoral.fr}}

\smallskip

{\small \it $^c$Department of Mathematical Sciences}\\
{\small \it University of Copenhagen, Denmark}\\
{\small \it and Centre de Recerca Matem\`atica (Barcelona)}\\
{\small \it email: {\rm kock@math.ku.dk}}

\smallskip

{\small \it $^d$Universit\'e C\^ote d'Azur, CNRS, UMR 7351}\\ 
{\small \it Laboratoire J.A.Dieudonn\'e}\\
{\small \it Parc Valrose, 06108 Nice Cedex 02, France.}\\
{\small \it email: {\rm patras@unice.fr}}}

\begin{document}

\date{}

\maketitle

\begin{abstract}
  Higher-order notions of Kreweras complementation have appeared in the
  literature in the works of Krawczyk, Speicher, Mastnak, Nica, Arizmendi,
  Vargas, and others. While the theory has been developed primarily for specific applications in free probability, it also possesses an elegant, purely combinatorial core that is of independent interest.
  The present article aims at offering a simple account of various aspects
  of higher-order Kreweras complementation on the basis of elementary
  arithmetic, (co)algebraic, categorical and simplicial properties of
  noncrossing partitions. The main idea is to consider noncrossing partitions as providing
  an interesting noncommutative analogue of the
  interplay between the divisibility poset and the multiplicative monoid of positive integers.
  Just as the divisibility poset can be regarded as the decalage of the
  multiplicative monoid, we exhibit the lattice of noncrossing partitions as
  the decalage of a partial monoid structure on noncrossing partitions
  encoding higher-order Kreweras complements. While our results may be considered familiar, several of the viewpoints can be regarded as novel, offering an efficient approach both conceptually and computationally.
\end{abstract}

\tableofcontents

%%%%%%%%%%%%%%%
%%%%%%%%%%%%%%%

\section{Introduction}
\label{sec:intro}

Although noncrossing partitions were initially studied out of combinatorial interest \cite{Kreweras:1972}, they have since found applications in many areas of mathematics and have been extensively studied \cite{Simion:2000}, \cite{McCammond}. They are enumerated by
Catalan numbers, ``probably the most ubiquitous sequence of numbers in
mathematics'' \cite{Stanley:Catalan}, and can be put in bijection with many
families of combinatorial objects,
% having the same generating series, 
such as binary trees, plane binary trees, and Dyck paths, to mention a few.
The set of noncrossing partitions of a linear order $[n] := \{1,2,\ldots,n\}$
carries an important lattice structure, which has been central to most 
applications of noncrossing partitions. However, considering them
together for all $n$, noncrossing partitions also carry shuffle-algebra 
structures~\cite{EbrahimiFard-Patras:1409.5664} (different from the one investigated in the present article), as well as two operad
structures~\cite{EbrahimiFard-Foissy-Kock-Patras:1907.01190}, that further 
enrich the theory. All this structure can be transported to other instances 
of Catalan combinatorics.

One feature characteristic for noncrossing partitions is the notion of
Kreweras complement, introduced by Kreweras in his foundational
article~\cite{Kreweras:1972}. The Kreweras complement defines an
automorphism of the set of noncrossing partitions of $[n]$, which is not an
involution but has period $n$.
The standard way to define the Kreweras complement of a
noncrossing partition $\alpha$ of  $[n]$
is to embed this set as the odd elements of $[2n]$ and look for the
noncrossing partition $\beta$ of the set of even integers in $[2n]$ such
that the union $\alpha\cup \beta$ is a noncrossing partition of $[2n]$ and is
maximal among such partitions. In a more algebraic notation, $\beta$ solves
the equation
$$
	(\alpha\shuffle_n\beta)\vee \{\{1,2\},\ldots,\{2n-1,2n\}\}=\{\{1,\ldots,2n\}\},
$$
for which $\alpha\shuffle_n\beta$, the perfect shuffle of $\alpha$ and
$\beta$ (see Definition~\ref{perfsh}), is required to
be noncrossing --- whereas $\vee$ is the join in the lattice of noncrossing
partitions of $[2n]$ (rigorous definitions will be given later).

In the present paper we are interested in higher versions of Kreweras
complement, motivated by applications in free probability (as very briefly
indicated in Section~\ref{sec:context} below). We regard these higher
Kreweras complements as a way to provide a noncommutative generalization of
certain features of the positive integers, more precisely the interplay
between the divisibility poset $\divposet$ and the multiplicative monoid
$\multmonoid$. Both feature notions of incidence (co)algebras and M\"obius
inversion. For the poset, this is the standard theory initiated by
Rota~\cite{Rota:Moebius}; for monoids the analogous constructions were
introduced by Cartier and Foata~\cite{Cartier-Foata}.

  The relationship between the two approaches can be formulated elegantly
  using the fact that posets and monoids are both examples of categories:
  Content, Lemay and Leroux~\cite{Content-Lemay-Leroux} observed that the
  assignment $a\divides b \mapsto b/a$ constitutes a functor from the
  category $\divposet$ to the category $\multmonoid$, and that this functor
  is CULF (``conservative'' and possessing ``unique lifting of
  factorizations''). They also identified the CULF functors as those that induce
  coalgebra homomorphisms at the level of incidence coalgebras. This
  coalgebra homomorphism is precisely the one from the (raw) incidence
  coalgebra of the divisibility poset to the reduced incidence coalgebra,
  where two `intervals' $a\divides b$ and $a'\divides b'$ in the divisibility
  poset are identified when $b/a = b'/a'$. This is important as it is 
  usually the setting of posets that is used to stage the whole theory, 
  whereas it is rather the reduced incidence algebra (which is the 
  incidence algebra of the monoid) that actually matters for M\"obius 
  inversion and related phenomena and tools.
  
  It was observed more recently~\cite{Galvez-Kock-Tonks:1612.09225} that the
  CULF map of Content--Lemay--Leroux is actually induced by decalage of
  simplicial sets: the (nerve of the) divisibility poset is the lower decalage
  of the (bar complex of the) multiplicative monoid; it is a general fact that
  the map back from a decalage of a simplicial set is CULF whenever the 
  simplicial set is a decomposition space, a class of simplicial sets that 
  contains nerves of categories (and in particular posets and monoids).

  We show that all these features carry over to the noncommutative setting
  of noncrossing partitions. Precisely, we exhibit the lattice of noncrossing
  partitions as the lower decalage of a simplicial set obtained by
  defining a suitable composition product on noncrossing partitions. The only 
  caveat, and likely the reason why this composition product has remained under the radar until now, is that it does not define a genuine monoid but rather a partial monoid in the sense of Segal~\cite{Segal:1973}. 
However, partial monoids are examples of decomposition spaces \cite{Bergner-Osorno-Ozornova-Rovelli-Scheimbauer:1609.02853}, and the theory of incidence (co)algebras and M\"obius inversion applies to decomposition spaces just as it does to posets, monoids, and categories \cite{Galvez-Kock-Tonks:1512.07573}.  
  The incidence coalgebra of this partial monoid is of some importance in
  free probability: the corresponding convolution algebra contains the multiplicative functions used in Speicher's free convolution
  (see Nica--Speicher~\cite{Nica-Speicher:Lectures}, Lecture 18).

\ \ \par

The article is organized as follows. Section~\ref{sec:context} briefly
gives some motivation and elements of context for the current main
application domain for our developments: free probability.
Section~\ref{sec:as} lists various algebraic structures on noncrossing
partitions, culminating with the definition of an arithmetic-inspired
composition product that turns out to encode all the information of Kreweras
complementation and its higher generalizations.
We will show, as an
application, how various key results of the theory can be formulated in 
terms of this partial monoid structure.
Section~\ref{sec:class} gives an account of classical incidence (co)algebras of
positive integers, with categorical and simplicial interpretations. 
The final Section~\ref{sec:coalg} develops coalgebraic,
categorical and simplicial properties of noncrossing partitions, showing
that they behave as a noncommutative version of the integers with respect
to M\"obius-inversion type calculus. We also briefly investigate
coalgebraic properties of $k$-divisible noncrossing partitions.

\begin{notation}
  The set $\{1,\ldots,n\}$ is denoted $[n]$. We use the rationals $\Q$ as
  ground field for our vector spaces. %set in sans serif typeface.

\end{notation}

%%%%%%%%%%%%%%%
%%%%%%%%%%%%%%%

\section{Context and motivation}
\label{sec:context}

Higher-order Kreweras complements can be defined by generalizing the definition we have recalled earlier, replacing the odd/even embedding of $[n]$ into $[2n]$ by the analogous embedding into $[kn]$. 
That these notions are meaningful is supported by combinatorial results in 
probability theory appearing in the works of  Krawczyk, Mastnak, Nica and 
Speicher \cite{Krawczyk-Speicher}, \cite{Nica-Speicher:Lectures}, \cite{Mastnak-Nica:TAMS2010}. Our work was initially motivated by the properties of the distributions of products of random variables in free probability -- specifically, multiplicative convolution. Similarly, our previous, technically independent article \cite{EbrahimiFard-Foissy-Kock-Patras:1907.01190} was driven by the properties of sums of random variables in free probability and additive convolution.

Recall from \cite{Nica-Speicher:Lectures} that a noncommutative probability
space is a pair $(A,\phi)$ consisting of an associative algebra $A$ and a unital linear
form $\phi$ on $A$. Free cumulants are multilinear
maps $\kappa_n$ from $A^{\otimes n}, n\in\N^\ast$, to $\Q$ (in free probability one would usually take $\C$ as a ground field, but  this choice has no relevance for the matters discussed in the present article and we stick therefore to $\Q$) defined by
induction (or M\"obius inversion) in the lattice of noncrossing partitions through the ($n$th-order) free moment-cumulant relation
$$
	\phi(a_1\cdots a_n)=\sum\limits_{\pi \in \parti(n)}\kappa_\pi(a_1,\ldots,a_n).
$$
Here, $\parti(n)$ stands for the set of noncrossing partitions of $[n]$ 
and $\kappa_\pi$ denotes the multiplicative extension of free cumulants 
to noncrossing partitions, that is, if $\pi=\{\pi_1,\ldots, \pi_k\} \in \parti(n)$, 
then
$$
\kappa_\pi(a_1,\ldots,a_n):=\prod\limits_{i=1}^k\kappa_{\pi_i}(a_1,\ldots,a_n)
:=
\prod\limits_{i=1}^k
% $$ 
% where  
% $$
% 	\kappa_{\pi_i}(a_1,\dots,a_n):=
	\kappa_{\norm{\pi_i}}(a_{n_1^i},\dots,a_{n_{\norm{\pi_i}}^i}),
$$ 
for $\pi_i=\{n_1^i,\ldots,n_{\norm{\pi_i}}^i\}$. Analogous to cumulants in classical probability, free cumulants in free probability characterize free independence, which is a good notion of independence in noncommutative probability theory \cite{BGhSchur02,Muraki02}: subalgebras $A_1,\dots,A_p$ of $A$ are freely independent if and only if the free cumulants $\kappa_n(a_1,\dots,a_n)$ vanish whenever at least two elements $a_i$ belong to different subalgebras in $A_1,\ldots,A_p$.

One motivation for the present work
is the following result
connecting computations in free probability with Kreweras complements, 
a consequence of \cite[Thm.~1.12]{Nica-Speicher:Lectures}: for
free cumulants of products of random variables we have
\begin{equation}
\label{eq1}
	\kappa_m ((a_1 \!\cdots a_p) ,\ldots, (a_{p(m-1)+1} \!\cdots a_{pm})) 
	=\sum\limits_{\pi\in \partipres{p}(m) \atop \hat\pi=\{[pm]\}}\kappa_\pi (a_1 ,\ldots , a_{pm}),
\end{equation}
where 
\begin{enumerate}
\item for $1\leq j\leq p$, $a_j,a_{p+j}, \dots ,a_{p(m-1)+j}$, $i=0,\dots,m-1$, belong to $A_j$, where $A_1,\dots ,A_p$ are freely independent subalgebras of $A$,
\item in the summation on the right-hand side of \eqref{eq1}, $\partipres{p}(m)$ denotes the set of $p$-preserving noncrossing partitions, i.e.~such that $i,j$ can be in the same block only if $i=j \mod p$,
\item the partition $\hat\pi$ stands for the finest noncrossing partition which is coarser than $\pi$ and such that each set $\{pi+1,\dots,p(i+1)\}$ is a subset of a single block of $\hat\pi$ for all $i=0,\ldots,m-1$.
\end{enumerate}

Below, we will not make further references to free probability. However, we point out that
formula \eqref{eq1} will drive our constructions on partitions $\pi \in
\partipres{p}(m)$ such that $\hat\pi=\{[pm]\}$. They are called $p$-completing in the
literature \cite[Def.~2.1]{Arizmendi-Vargas:2012} and provide a natural
$p$-fold generalization of classical Kreweras complementation. When $p=2$,
the sum in (\ref{eq1}) is indeed equivalently formulated over pairs of a noncrossing
partition and its Kreweras complement, as described previously in this
introduction --- see also Lemma~\ref{Kco2} below.

%%%%%%%%%%%%%%%
%%%%%%%%%%%%%%%

\section{Algebraic structures}
\label{sec:as}

We shall now go through a series of algebraic and coalgebraic structures on
noncrossing partitions. The lattice structure \ref{ssec:os}, which goes
back to Kreweras~\cite{Kreweras:1972}, is the usual way to approach
noncrossing partitions; the Nica--Speicher
Lectures~\cite{Nica-Speicher:Lectures} is a standard reference for this. The
Kreweras complementation~\ref{ssec:Kc} is also well
studied~\cite{Nica-Speicher:Lectures}. The power maps \ref{ssec:power}, the
concatenation product~\ref{ssec:concat}, the complete shuffle
product~\ref{ssec:perfect} are introduced here for technical purposes, together 
with the idea that there is a proper ``arithmetic'' behavior of
noncrossing partitions. We are not aware of specific references for them, but they are anyway based on elementary and classical constructions (in combinatorics, cards shufflings...); 
and they are of
course closely related to well-known structures on noncrossing partitions (see in particular 
Arizmendi--Vargas~\cite{Arizmendi-Vargas:2012}). 
The composition product \ref{ssec:partialmonoid} was studied by
Biane~\cite{Biane:1997} (under the trace map into the permutation groups, in 
fact mentioned already by Kreweras~\cite{Kreweras:1972}), but without
noticing that it defines a partial monoid. Generally, the utility of partial
monoids \ref{ssec:partialmonoid} does not seem to have been well appreciated outside 
algebraic topology~\cite{Segal:1973}, and the
crucial fact that partial monoids have incidence coalgebras (which we come to 
in \ref{ss1}) is quite
recent~\cite{Galvez-Kock-Tonks:1512.07573},
\cite{Bergner-Osorno-Ozornova-Rovelli-Scheimbauer:1609.02853}.

Recall (from \cite{Kreweras:1972}, \cite{Nica-Speicher:Lectures})
that a partition $\pi=\{\pi_1,\dots , \pi_l\}$ of $[n]$ is 
noncrossing, i.e., $\pi\in \parti(n)$, if and only if there are no distinct blocks
$\pi_i$ and $\pi_j$ such that $$\exists a,c\in \pi_i,\ b,d\in \pi_j \mid 
a<b<c<d.$$ 

Noncrossing partitions of an arbitrary totally ordered set $S$ are defined
similarly, and form a lattice $\parti(S)$. An increasing bijection
$\phi:S\isopil T$ induces a bijection $\parti(S)\isopil\parti (T)$ denoted
$\parti(\phi)$. In particular, the integer translation by $p$ of an element
$\pi$ of $\parti(n)$ is a noncrossing partition of
$[n]+p:=\{1+p,\ldots,n+p\}$ that we denote $\pi+p$. Similarly, the dilation
by $p$ of an element $\pi$ of $\parti(n)$ is a noncrossing partition of
$p\cdot [n]:=\{p,2p,\ldots,np\}$ that we denote $p\cdot \pi$.
Finally, given an arbitrary totally ordered finite set $S$ of cardinality
$n$, and given a noncrossing partition $\beta\in\parti(S)$, we write
$\st(\beta)$ for the noncrossing partition in $\parti(n)$ obtained by
transporting $\beta$ along the unique increasing bijection between $S$ and
$[n]$. For example, $\st(\{\{1,8\},\{3,5\}\})=\{\{1,4\},\{2,3\}\}$.

The blocks of a noncrossing partition, $\pi=\{\pi_1,\dots , \pi_l\}$, are
ordered by $\pi_i \preceq \pi_j$ if and only if $\min(\pi_j) \leq \min(\pi_i) \leq
\max(\pi_i) \leq \max(\pi_j)$. (That is, the block $\pi_i$ is equal to $\pi_j$ or nested inside it if $\min(\pi_j) < \min(\pi_i) \leq
\max(\pi_i) < \max(\pi_j)$.)

A noncrossing partition in $\parti(n)$ is 
{\it irreducible} 
if and only if it has a unique maximal block for the order $\prec$, that
is, if $1$ and $n$ belong to the same block. In general, the 
irreducible
components of a noncrossing partition are the subsets of $\pi$ of the
form $\{C \mid C\preceq B\}$, where $B$ is a maximal block. The same
definition extends to $\parti(S)$. For example, the 
irreducible
components of
$\{\{1,3\},\{2\},\{4,8\},\{5,6,7\}\}$ are $\{\{1,3\},\{2\}\}$ and
$\{\{4,8\},\{5,6,7\}\}$. The 
irreducible
components of
$\{\{1,4\},\{2\},\{5,10\},\{6,7\}\}$ are $\{\{1,4\},\{2\}\}$ and
$\{\{5,10\},\{6,7\}\}$.

%%%%%%%%%%%%%%%

\subsection{Order structure}
\label{ssec:os}

The set $\parti(S)$ has a natural partial order of coarsening
which we write using notation borrowed from arithmetic: $\pi \divides \mu$ if and only
if every block of $\pi$ is contained in a block of $\mu$. The opposite order is the refinement order: $\mu$ is coarser than $\pi$, that is finer than $\mu$.
Coarsening makes
$\parti(S)$ a lattice. The meet and join are denoted as usual $\pi\wedge
\mu$ and $\pi\vee\mu$, respectively. 
The meet of two noncrossing partitions agrees with their meet as partitions; this follows directly from the computation of the meet of partitions whose blocks are intersections of blocks of the two partitions whose meet is taken.
The construction of the join of two noncrossing partitions $\alpha$ and $\beta$ is slightly more complex: one can for example take their join $\gamma$
in the lattice of partitions and construct $\alpha\vee\beta$ as the noncrossing closure of $\gamma$ (the noncrossing closure of $\gamma$ is the 
smallest noncrossing partition larger than $\gamma$ in the partition lattice, it can be concretely obtained from $\gamma$ for example by recursively merging the blocks in $\gamma$ that cross each other).

The minimal partition (whose blocks are
singletons) is written $0_S$ or simply $0_n$ when $S=[n]$. The maximal
element (with only a single block) is written $1_S$, respectively $1_n$ if $S=[n]$.
 
As for any locally finite poset (see Rota~\cite{Rota:Moebius}), we can associate to $\parti(n)$ an incidence (co)algebra. 

\begin{defi}[Raw incidence (co)algebra]\label{incidal}
  The incidence coalgebra $\incCoalg{\parti(n),\divides\,}$ of the poset $(\parti(n),\divides\,)$ is 
  spanned as a vector space by the intervals $[\alpha,\gamma]$ (consisting 
  of all $\beta$ with $\alpha\divides\beta\divides\gamma$), and with 
  comultiplication given by
  $$
 	\Delta([\alpha,\gamma]) = \sum_{\alpha\divides\beta\divides\gamma} 
	[\alpha,\beta] \otimes [\beta,\gamma]
  $$
  and counit the ``Kronecker delta'': $\partial([\alpha,\gamma]) = 1$ for $\alpha=\gamma$, and zero else.
  
The (raw) incidence algebra $\incAlg{\parti(n),\divides\,}$ is given by linear functions on $\incCoalg{\parti(n),\divides\,}$
with multiplication given by the convolution product
$$
  	(f\ast g)(\alpha,\gamma):=\sum\limits_{\alpha\divides \beta\divides \gamma}f(\alpha,\beta)g(\beta,\gamma) 
$$
and unit $\partial$, where, for notational simplicity we have abbreviated $f([\alpha,\gamma])$ to $f(\alpha,\gamma)$.
\end{defi}
 
The lattice structure of noncrossing partitions has several interesting
properties, among which one will be useful later in this article. Consider
the set $\parti_\pi(n)$ of noncrossing partitions
$\mu\in\parti(n)$ containing a given set of disjoint (and noncrossing) blocks
$\pi_1,\ldots,\pi_k$ with $\pi_i=\{x_1^i,\ldots,x_{\norm{\pi_i}}^i\}$. The order on $\parti(n)$ restricts to an order on $\parti_\pi(n)$. 

\begin{lemma}\label{relatlatt}
 The poset $\parti_\pi(n)$ is
  a sublattice of $\parti(n)$. It is isomorphic as a lattice to a
  cartesian product of lattices $\parti(S_j)$ where the $S_j$ form a partition of $[n]-\bigcup\limits_{i=1}^k\pi_i$.
\end{lemma}

Indeed, write $\pi_{\operatorname{min}}$ for the minimal partition in the poset $\parti_\pi(n)$: it is the
noncrossing partition containing the $\pi_i$ and the singletons $\{x\}$,
where $x$ runs over the elements of $[n]$ that are not contained in the blocks $\pi_i$. 
For any such $x$, write $\pi_x:=\min\{\pi_i,\{x\}\prec\pi_i\}$ if the set $\{\pi_i,\{x\}\prec\pi_i\}$
is non empty and $\pi_x:=\emptyset$ else. That $\pi_x$ is well defined follows from general elementary properties
of noncrossing partitions: if there are blocks $\pi,\pi',\pi''$ of a 
noncrossing partition such that $\pi\prec \pi',\pi''$, then either 
$\pi\prec\pi'\prec\pi''$
or $\pi\prec\pi''\prec\pi'$. This can be deduced easily from the fact that $\min(\pi')<\min(\pi)\leq\max(\pi)<\max(\pi')$, that
$\min(\pi'')<\min(\pi)\leq\max(\pi)<\max(\pi'')$, and that $\pi'$ and $\pi''$ do not cross. We call this property the tree-ordering property of blocks.

Write now $\kappa_i:=\{x|\pi_x=\pi_i\}$ and $\kappa_0:=\{x|\pi_x=\emptyset\}$.
It follows from the tree-ordering property of blocks that if $\mu\in\parti_\pi(n)$ and $\zeta$ is a block of $\mu$ not a $\pi_i$, it is then contained
in a $\kappa_i,\ i\in\{1,\cdots,k\}$ or in $\kappa_0$. Furthermore, if $\zeta$ is contained in $\kappa_i$, then $\zeta\prec \pi_i$ (now for the ordering
of blocks in $\mu$) and therefore, as $\mu$ is noncrossing, there exists a unique $l,\ 1\leq l<|\pi_i|$ such that
$x^i_l<x<x^i_{l+1}$ for $x\in\zeta$. 

Let us now write $\kappa_i^l$ for $\{x|\pi_x=\pi_i \  \& \ x^i_l<x<x^i_{l+1}\}$. In conclusion, we get:
$\mu \in\parti_\pi(n)$ if and only if it is a noncrossing partition such that
\begin{itemize}
\item it contains the $\pi_i$ as blocks;
\item each of its blocks is contained in one of the subsets $\kappa_0,\kappa_1^1,\cdots,\kappa_1^{|\pi_1|-1},\cdots,\kappa_k^1,\cdots,\kappa_k^{|\pi_k|-1}$.
\end{itemize}

The maximal element in $\parti_\pi(n)$ is then given by
$$
	\pi_{\operatorname{max}}:=
	\{\pi_1,\ldots,\pi_k,\kappa_0\}\cup\bigcup\limits_{1\leq i\leq k}\{\kappa_i^1,\ldots,\kappa_i^{\norm{\pi_i}-1}\}.
$$
In general elements $\mu\in\parti_\pi(n)$ are 
noncrossing partitions obtained as the union of $\{\pi_1,\ldots,\pi_k\}$ with noncrossing partitions of $\kappa_0$ and the $\kappa_i^j$. 
This concludes the proof of the Lemma.

%%%%%%%%%%%%%%%
 
\subsection{Power maps}
\label{ssec:power}

Consider the `coface maps'
\begin{eqnarray*}
	f_i^n:[n] 	& \longrightarrow 	& [n+1]  \\
	j 		& \longmapsto 		& \begin{cases} j & \text{ for } j\leq i \\
	j+1 		& \text{ for } j> i ,
  \end{cases}
\end{eqnarray*}
where $i=0,\dots, n$.

The $i$-th {\em replication map} (for $i=1,\dots ,n$), denoted $r_i^n$ is the map from $\parti(n)$ to
$\parti(n+1)$ defined by sending a non crossing partition
$\pi=(\pi_1,\ldots,\pi_l)$ to $r_i^n(\pi)=\{\pi'_1,\dots , \pi'_l\}$, where
$\pi'_j:=f^n_i(\pi_j)\cup\{i+1\}$ for $i\in \pi_j$, and $\pi'_j:=f^n_i(\pi_j)$
for $i\notin \pi_j$.  In words, a copy of $i$ is created, labeled $i+1$ and put in the same
block as $i$, and the elements above $i$ are translated by $+1$.

\begin{defi}
  The {\em $p$-th power} of an element $\pi=\{\pi_1,\ldots , \pi_l\}$ of
  $\parti(n)$ is the element of $\parti(pn)$ obtained as
  $\pi^p=(r_1^{pn-1}\circ\cdots \circ r_1^{1+(n-1)p})\circ\cdots\circ
  (r_n^{n+p-2}\circ\cdots\circ r_n^n)(\pi)$.
  In words, each element $i$ is replicated $p$ times and all the replicas are
  put in the same block as $i$, integer labeling being changed in a coherent way.
\end{defi}

For example (with $n=4$ and $p=2$), if $\pi=\{\{1,4\},\{2,3\}\}$ then
$\pi^2=\{\{1,2,7,8\},\{3,4,5,6\}\}$.

One obtains the arithmetic rule 
$$
	(\pi^p)^q=\pi^{pq}.
$$

Let us immediately state an obvious but useful characterization of the
image of the power map: a noncrossing partition $\alpha\in \parti(pn)$ is in the image of
the $p$-th power map if and only if each of the sets $\{1,\ldots, p\}$; $\{p+1,\ldots, 2p\}$;
$\cdots$; $\{pn-p+1,\ldots,pn\}$ is contained in a block of $\alpha$. In more
abstract (but equivalent) terms:

\begin{lemma}
\label{powers}
  A noncrossing partition $\alpha\in\parti(pn)$ is in the image of the
  $p$-th power map if and only if
  $$
	 \alpha=\alpha\vee\{\{1,\ldots, p\}; \{p+1,\ldots, 2p\}; \cdots; \{pn-p+1,\ldots,pn\}\}.
  $$
\end{lemma}
We will write $\pi=\sqrt{\pi^2}$ and more generally $\pi=(\pi^p)^{\frac{1}{p}}$. Notice that the $p$-th root operation is not defined for general noncrossing partitions in $\parti(pn)$: by Lemma \ref{powers}, $\alpha^{\frac{1}{p}}$ is defined if and only if $\alpha=\alpha\vee\{\{1,\ldots, p\}; \{p+1,\ldots, 2p\}; \cdots; \{pn-p+1,\ldots,pn\}\}.$

%%%%%%%%%%%%%%%
 
\subsection{Concatenation product}
\label{ssec:concat}

The concatenation of two noncrossing partitions, $\alpha \in \parti(n)$ and $\beta\in \parti(m)$, is the noncrossing partition $\alpha \cdot \beta:=\alpha\cup(\beta+n)$. It is an easy exercise to check the next

\begin{lemma}
  $NC := \coprod_{n\in\N} \parti(n)$ with the concatenation product is the 
  free monoid on the set of 
 irreducible noncrossing partitions.
\end{lemma}

One also obtains the (noncommutative) arithmetic rule 
$$
	(\alpha\cdot \beta)^p=\alpha^p\cdot \beta^p,
$$
where power maps are defined as in the previous section (and not as powers for the concatenation product).

\begin{remark}
  Let $\mathsf{NC}$ denote the linear span on $NC$. Equipped with the concatenation product, it is a free associative algebra over the set of irreducible
  noncrossing partitions. It can be equipped with a cocommutative Hopf algebra
  structure by letting the irreducible noncrossing partitions be primitive
  elements. The construction is natural and allows one to relate the theory of
  noncrossing partitions to the theory of free Lie algebras in a canonical way
  (see e.g. \cite{Cartier-Patras:2021} for definitions and details on Hopf algebras and free Lie algebras).
\end{remark}

%%%%%%%%%%%%%%%

\subsection{Perfect shuffle product}
\label{ssec:perfect}

\begin{defi}\label{perfsh}
 Let $\alpha\in \parti(kn)$ and $\beta\in\parti(ln)$. The $n$-{\em perfect shuffle}  of $\alpha$ and $\beta$ 
\begin{equation}
\label{perfshuffle} 
 	\alpha\ast_n\beta:=i_k(\alpha)\cup e_l(\beta),
\end{equation}
is a partition of $[(k+l)n]$, defined in terms of the monotone embeddings 
  \begin{align*}
  i_k & :
  \left\{\begin{array}{rcl}
  [kn]&\longrightarrow&[(k+l)n]\\
  ak+j&\longmapsto&a(k+l)+j, \quad\quad\ 
  \text{ for } 0\leq a\leq n-1,\ 1\leq j\leq k
  \end{array}\right.
  \\
  e_l&:\left\{\begin{array}{rcl}
  [ln]&\longrightarrow&[(k+l)n]\\
  \,al+j&\longmapsto&a(k+l)+k+j
  \quad \text{ for } 0\leq a\leq n-1,\ 1\leq j\leq l.
  \end{array}\right.
  \end{align*}
Notice that $\alpha\ast_n\beta$ is not a noncrossing partition in general.
\end{defi} 

For example (with $n=3$, $k=2$, and $l=1$), the 3-perfect shuffle of
$\alpha= \{\{1,5,6\},\{2,4\},\{3\}\}$ with $\beta=\{\{1\},\{2,3\}\}$ is $\{\{1,7,8\}
,\{2,5\}, \{4\},\{3\},\{6,9\}\}$ (it is not a noncrossing partition). For example, the integer $5=2*2+1$ in the first block of $\alpha$ is sent to $7=2*3+1$, whereas the integer $2=1*1+1$ in the second block of $\alpha$ is sent to $6=1*3+2+1$.

 In
the following picture, on top and below are indicated the initial values of
the elements, before they are shuffled and relabeled.
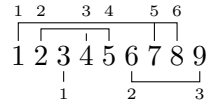
\[
\begin{tikzpicture}

%\footnotesize
\draw (0.0,0) node {$1$};
\draw (0.3,0) node {$2$};
\draw (0.6,0) node {$3$};
\draw (0.9,0) node {$4$};
\draw (1.2,0) node {$5$};
\draw (1.5,0) node {$6$};
\draw (1.8,0) node {$7$};
\draw (2.1,0) node {$8$};
\draw (2.4,0) node {$9$};

\begin{scope}[shift={(0,0.2)}]
\draw (0,0.0) -- (0,0.24) -- (2.1,0.24) -- (2.1,0);
  \draw (1.8,0)--(1.8,0.24);
\draw (0.3,0.0) -- (0.3,0.16) -- (1.2,0.16) -- (1.2,0);
\draw (0.9,0.0) -- (0.9,0.1);
\tiny
\draw (0.0,0.4) node {$1$};
\draw (0.3,0.4) node {$2$};
\draw (0.9,0.4) node {$3$};
\draw (1.2,0.4) node {$4$};
\draw (1.8,0.4) node {$5$};
\draw (2.1,0.4) node {$6$};
\end{scope}

\begin{scope}[shift={(0,-0.2)}]
\draw (1.5,0.0) -- (1.5,-0.14) -- (2.4,-0.14) -- (2.4,0);
\draw (0.6,0.0) -- (0.6,-0.14);
\tiny
\draw (0.6,-0.3) node {$1$};
\draw (1.5,-0.3) node {$2$};
\draw (2.4,-0.3) node {$3$};
\end{scope}

\end{tikzpicture}
\] 

The $n$-perfect shuffle product is easily seen to be associative. Given $\alpha\in\parti(n)$, it also 
satisfies
$$
	\alpha^p=\alpha\ast_n\cdots \ast_n\alpha \vee\{\{1,\ldots, p\}; \{p+1,\ldots, 2p\}; \cdots; \{pn-p+1,\ldots,pn\}\}
$$
where the product $\alpha\ast_n\cdots \ast_n\alpha$ on the right-hand side has $p$ factors.

\begin{defi}
\label{def:admissible}
With the same notation as above, when (and only when) the $n$-perfect shuffle \eqref{perfshuffle} is a noncrossing partition, that is, $i_k(\alpha)\cup e_l(\beta)  \in \parti((k+l)n)$, we say that the pair $(\alpha,\beta)$ is $n$-{\it admissible} and, to notationally distinguish that case, set
  $$
  	\alpha\shuffle_n\beta:=\alpha\ast_n\beta.
  $$
When $(\alpha,\beta)\in\parti(n)^2$, we will slightly abusively say that the 
pair is admissible for ``the pair is $n$-admissible''.
\end{defi}

\begin{defi}
  More generally, given
  $\alpha_1\in\parti(k_1n),\ldots,\alpha_p\in\parti(k_pn)$, we say that the
  $p$-tuple $(\alpha_1,\ldots,\alpha_p)$ is {$n$-\it admissible} if and only if
  $\alpha_1\ast_n\cdots \ast_n\alpha_p$ is noncrossing, in which case we
  also write $\alpha_1\ast_n\cdots \ast_n\alpha_p=:\alpha_1\shuffle_n\cdots
  \shuffle_n\alpha_p.$

When $(\alpha_1,\cdots,\alpha_k)\in\parti(n)^k$, we will slightly abusively say 
that the  $k$-tuple is admissible for ``the $k$-tuple is $n$-admissible''.
\end{defi}

\begin{lemma}
\label{admissibility}
  With the same notation, a $p$-tuple $(\alpha_1,\ldots,\alpha_p)$ is $n$-admissible
  if and only if all pairs $(\alpha_i,\alpha_j)$ with $1\leq i<j\leq p$ are
  $n$-admissible. 
%Equivalently, a $p$-tuple $(\alpha_1,\ldots,\alpha_p)$ is not $n$-admissible if and only if there exists a pair $(\alpha_i,\alpha_j)$ with $1\leq  i<j\leq p$ that is not $n$-admissible.
\end{lemma}

\begin{proof}
  The property for a partition to be noncrossing depends only on the
  pairwise behavior of its blocks. As, by construction, the blocks of $\alpha_1\ast_n\cdots
  \ast_n\alpha_p$  are all obtained from and in bijection
  with the blocks of the $\alpha_i$ and as furthermore the $\alpha_i$ are noncrossing
  partitions, it is enough to test the noncrossing property considering
  only the relative positions of blocks obtained from a $\alpha_i$ and from a $\alpha_j$
  for $i<j$, that is to test if $\alpha_i\ast_n\alpha_j$ is noncrossing for
  $i<j$.
\end{proof}

We list (without proofs) some elementary properties of $n$-admissibility and of the $n$-perfect shuffle product:

\begin{lemma}
\label{admcond}
  If $(\alpha, \beta)$ in $\parti(kn)\times \parti(ln)$ is $n$-admissible and $\alpha'\divides \alpha$,
  $\beta'\divides \beta$, then $(\alpha',\beta')$ is $n$-admissible. 
%   Conversely,
  If $(\alpha,\beta)$ is not $n$-admissible (that is,
  $\alpha\ast_n\beta \notin \parti((k+l)n))$ and $\alpha\divides \alpha'$,
  $\beta\divides \beta'$, then $(\alpha',\beta')$ is not $n$-admissible.
\end{lemma}

\begin{lemma}
\label{increa}
  The $n$-perfect shuffle product is increasing: if $(\alpha, \beta)$,
  $(\alpha',\beta')$ in $\parti(kn)\times \parti(ln)$ are $n$-admissible with $\alpha\divides \alpha'$ and
  $\beta\divides \beta'$, then
  $$
  \alpha{\shuffle_n}\beta \divides \alpha'{\shuffle_n}\beta'.
  $$
\end{lemma}

\begin{lemma}\label{admis-is-lattice}
  Given $\alpha\in \parti(n)$, the set of noncrossing partitions $\beta\in \parti(n)$
  such that $(\alpha,\beta)$ is admissible is ordered by coarsening. It is
  stable by meets and joins and forms a sublattice of the lattice of
  noncrossing partitions in $\parti(n)$.
\end{lemma}

\begin{proof}
  Apply Lemma~\ref{relatlatt} to the case where $\pi_1,\ldots,\pi_k$ is the
  set of blocks in the image $\tilde\alpha$ of $\alpha$ when $\alpha$ is embedded
  into $\parti(2n)$ as a partition of the set of odd elements.  The result
  follows. Notice that with these
  conventions, the lattice $\parti_\pi(2n)$ is the set of all
  $\alpha\shuffle_n\beta$, where $(\alpha,\beta)$ is admissible.
\end{proof}
 
\begin{defi} An element $\alpha$ of $\parti(kn)$ is said to be {\em $k$-preserving}  when it is
  the case that any two integers in $[kn]$ in the same block of $\alpha$ are
  equal modulo $k$, cf.~Arizmendi--Vargas~\cite{Arizmendi-Vargas:2012}.
  The set of $k$-preserving partitions in $\parti(kn)$ is written $\partipres{k}(n)$.\end{defi}

The following obvious Lemma allows to restate the definition in terms of perfect shuffles:
\begin{lemma}
 An element $\alpha$ of $\parti(kn)$ is {$k$-preserving} if and only
  if it can be written $\alpha_1\shuffle_n\cdots \shuffle_n\alpha_k$ with the
  $\alpha_i$ in $\parti(n)$. 
\end{lemma}

%%%%%%%%%%%%%%%
 
\subsection{The partial monoid structure}
\label{ssec:partialmonoid}

We have seen that the perfect shuffle \eqref{perfshuffle} of two noncrossing
partitions is not always noncrossing. This creates some difficulties to provide
a synthetic picture allowing to deal simultaneously with noncrossing partitions
as if they were at the same time the elements of a poset and of a monoid --- as
occurs with the divisibility poset and the multiplicative monoid of the
integers, see Subsection \ref{ssec:standardcons}.

\begin{defi}
\label{compp}
  The {\em composition product} on noncrossing partitions is the partially defined product defined for $\alpha,\beta\in \parti(n)$ such that
  $\alpha\ast_n\beta$ is noncrossing by
  \begin{equation}
  \label{compoprod}
  	\alpha\circ \beta:=\sqrt{(\alpha\shuffle_n\beta)\vee \{\{1,2\},\{3,4\},\ldots,\{2n-1,2n\}\}}.
  \end{equation}
\end{defi}

Notice that this product is well defined as a consequence of Lemma~\ref{powers}. Recalling Lemma \ref{increa}, the next Lemma shows that the composition product interacts well with the order structure:

\begin{lemma}\label{monoto}
  Given admissible pairs $(\alpha,\beta)$, $(\alpha',\beta)$,
  $(\alpha,\beta')$ with $\alpha\divides \alpha'$ and $\beta\divides
  \beta'$, we have
$$
	\alpha{\circ}\beta \divides  \alpha'{\circ}\beta \quad\text{ and } \quad \alpha{\circ}\beta \divides  \alpha{\circ}\beta'.
$$
Moreover, these inequalities are strict if $\alpha\neq \alpha'$, respectively $\beta\neq\beta'$.
\end{lemma}

The Lemma can be seen as a restatement, in algebraic language, of standard 
monotonicity properties of Kreweras complementation which can be found for example in \cite{Nica-Speicher:Lectures}. However, as it is interesting to
see how they translate into our framework, we sketch the proof:

\begin{proof}
  Since the composition product \eqref{compoprod} is clearly weakly increasing, it is enough to assume that $\alpha\neq \alpha'$ (respectively $\beta\neq\beta'$) and show that the
  number of blocks in $\alpha'\circ\beta$ and $\alpha\circ\beta'$ is
  strictly less than the number of blocks in $\alpha\circ\beta$. It is then
  also enough to study the particular case where $\alpha'$ has one block
  less than $\alpha$, or similarly for $\beta'$ and $\beta$.

  There are several possible configurations. Let us assume for example that
  $\beta'$ is obtained from $\beta$ by the merge of two blocks $\beta_i$
  and $\beta_{i+1}$ that are not comparable (for the coarsening ordering of
  blocks inside $\beta$) and that $\min(\beta_i)<\min(\beta_{i+1})$ (that
  is, the block indexed by $i$ is to the left of the block indexed by $i+1$). Such a configuration implies that the subinterval
  $[2\cdot \max(\beta_i)+1,2\cdot\max(\beta_{i+1})]$ of $[2n]$ is an union
  of blocks in $\alpha\shuffle_n\beta$ (otherwise one can show that
  $(\alpha,\beta')$ would not be admissible as the merge of $\beta_i$ and
  $\beta_{i+1}$ would create a crossing when moving from
  $\alpha\ast_n\beta$ to $\alpha\ast_n\beta'$). As $2\cdot \max(\beta_i)+1$
  is odd and $2\cdot\max(\beta_{i+1})$ is even, this in turn implies that
  $2\cdot \max(\beta_{i+1})$ does not belong to the same block as $2\cdot
  \max(\beta_i)$ in $(\alpha\shuffle_n\beta)\vee
  \{\{1,2\},\{3,4\},\ldots,\{2n-1,2n\}\}$. However,  as $\beta'$ is obtained from $\beta$ by merging $\beta_i$
  and $\beta_{i+1}$, they belong to
  the same block in $(\alpha\shuffle_n\beta')\vee
  \{\{1,2\},\{3,4\},\ldots,\{2n-1,2n\}\}$, so that the latter has at least
  one block less than the former, which concludes the proof of this case.
  The other cases can be obtained similarly.
\end{proof}

\begin{defi}
  A {\em partial monoid} (in the sense of Segal~\cite{Segal:1973}) is a set
  $M$ equipped with a partially-defined binary operation $M\times M \to M$
  required to be associative and unital. More precisely, one is given a
  subset $M_2 \subset M \times M$ and a function $M_2 \to M$ written with
  infix notation $(m_1,m_2) \mapsto m_1 \cdot m_2$ with the property that
  $(m_1 \cdot m_2) \cdot m_3$ is defined if and only if $m_1 \cdot
  (m_2\cdot m_3)$ is defined, and, in that case, the two expressions are
  equal. Finally there should be a neutral element $1$ such that $1\cdot m$
  and $m\cdot 1$ are defined and equal to $m$, for all $m\in M$.
\end{defi}

\begin{prop}\label{ispartial}
  The set $\parti(n)$ equipped with the partially-defined binary operation
  \eqref{compoprod} 
  %$\circ$ 
  from the set of admissible pairs to $\parti(n)$ is a partial
  monoid. Its unit is the noncrossing partition $0_n$.
\end{prop}

For the proof we shall use the following lemma. Its proof is omitted as it follows easily from the definitions; it illustrates some of the power of noncrossing arithmetic techniques:

\begin{lemma}\label{interids}
Let $\alpha,\beta\in \parti(n)$ and assume that they form an admissible pair. Then $\alpha^2\ast_n\beta$ and $\alpha\ast_n\beta^2$ are also noncrossing and the following identities hold:
\begin{align*}
	\alpha\circ\beta
	&=\sqrt{(\alpha\shuffle_n\beta)\vee \{\{1,2\},\{3,4\},\ldots,\{2n-1,2n\}\}}\\
	&=\left({(\alpha^2\shuffle_n\beta)\vee \{\{1,2,3\},\{4,5,6\},\ldots,\{3n-2,3n-1,3n\}\}}\right)^{\frac{1}{3}}\\
	&=\left({(\alpha^2\shuffle_n\beta)\vee \{\{2,3\},\{5,6\},\ldots,\{3n-1,3n\}\}}\right)^{\frac{1}{3}}\\
	&=\left({(\alpha\shuffle_n\beta^2)\vee \{\{1,2,3\},\{4,5,6\},\ldots,\{3n-2,3n-1,3n\}\}}\right)^{\frac{1}{3}}.
\end{align*}
\end{lemma}

\begin{proof}[Proof of Proposition~\ref{ispartial}]
  All pairs $(\alpha,0_n)$ and $(0_n,\alpha)$ being admissible, the fact that
  $0_n$ is a unit for the composition product $\circ$ is a direct consequence of its definition \eqref{compoprod} and is left as an exercise.

  Assume now that the triple $(\alpha,\beta,\gamma)$ is not admissible. By
  Lemma~\ref{admissibility}, this is equivalent to having at least
  one of the three pairs $(\alpha,\beta)$, $(\alpha,\gamma)$,
  $(\beta,\gamma)$ being not admissible, and therefore at least one of the
  three expressions $\alpha\circ\beta$, $\alpha\circ\gamma$ and
  $\beta\circ\gamma$ is not defined. From Lemma~\ref{admcond} and
  $\alpha,\beta\divides \alpha\circ\beta$; $\beta,\gamma\divides
  \beta\circ\gamma$, we get that both $(\alpha\circ\beta)\circ\gamma$ and
  $\alpha\circ(\beta\circ\gamma)$ are not defined.

  Assume finally that the triple $(\alpha,\beta,\gamma)$ is admissible, which is
  equivalent to assuming that the three pairs $(\alpha,\beta)$,
  $(\alpha,\gamma)$, $(\beta,\gamma)$ are admissible. 
  % 
  % Let us apply these identities to the admissible triple 
  % $(\alpha,\beta,\gamma)$. 
  Using associativity of joins in lattices and Lemma~\ref{relatlatt} we get:
  \begin{align*}
  \lefteqn{(\alpha\ast_n\beta\ast_n\gamma)\vee\{\{1,2,3\},\{4,5,6\},\ldots,\{3n-2,3n-1,3n\}\}}\\
  &=(\alpha\ast_n\beta\ast_n\gamma)\vee \{\{1,2\},\{3\},\{4,5\},\ldots,\{3n-2,3n-1\},\{3n\}\}\\
 &\ \ \ \ \ \ \ \ \ \ \ \ \ \ \ \ \vee\{\{1,2,3\},\{4,5,6\},\ldots,\{3n-2,3n-1,3n\}\}\\
  &=(((\alpha\ast_n\beta)\vee \{\{1,2\},\{3,4\},\ldots,\{2n-1,2n\}\})\ast_n\gamma)\vee\{\{1,2,3\},\{4,5,6\},\ldots,\{3n-2,3n-1,3n\}\}\\ 
  &=((\alpha\circ\beta)^2\ast_n\gamma)\vee\{\{1,2,3\},\{4,5,6\},\ldots,\{3n-2,3n-1,3n\}\},
\end{align*}
  so that, by applying Lemma~\ref{interids} we get
  $$
  	(\alpha\circ\beta)\circ\gamma=\left((\alpha\ast_n\beta\ast_n\gamma)\vee\{\{1,2,3\},\{4,5,6\},\ldots,\{3n-2,3n-1,3n\}\}\right)^{\frac{1}{3}}.
$$
  The same reasoning shows that 
  $$
  	\alpha\circ(\beta\circ\gamma)=\left((\alpha\ast_n\beta\ast_n\gamma)\vee\{\{1,2,3\},\{4,5,6\},\ldots,\{3n-2,3n-1,3n\}\}\right)^{\frac{1}{3}},
$$
  which implies $(\alpha\circ\beta)\circ\gamma=\alpha\circ(\beta\circ\gamma)$ and concludes the proof.
\end{proof}

We note that the same reasoning together with an inductive argument implies more generally:

\begin{lemma}
\label{translationlemma}
  Given $(\alpha_1,\ldots,\alpha_k)$ an admissible $k$-tuple, we have
  $$
  \alpha_1\circ\cdots \circ\alpha_k
  =\left((\alpha_1\ast_n\cdots\ast_n\alpha_k)\vee\{\{1,\ldots,k\},\ldots,\{kn-k+1,\ldots,kn\}\}\right)^{\frac{1}{k}}.
  $$
\end{lemma}

%%%%%%%%%%%%%%%

\subsection{The Kreweras automorphism}
\label{ssec:Kc}

In this subsection we give a brief account of
Kreweras complementation. This is a well-studied and classical subject and we
only hint at how it can be described in the algebraic formalism we have
introduced, omitting details of proofs that can be found (or easily adapted
from) the Nica and Speicher textbook \cite{Nica-Speicher:Lectures}.

Given a noncrossing partition $\alpha\in \parti(n)$, the set of partitions
$\beta$ such that $(\alpha,\beta)$ is admissible is ordered by coarsening. We
have seen 
in Lemma~\ref{admis-is-lattice}
that it is stable by meets and joins and
forms a sublattice of the lattice of noncrossing partitions $\parti(n)$. Its
maximal element is, by definition, the Kreweras complement of $\alpha$, denoted
$K(\alpha)$. Concretely, the latter is determined in the following way: if
$1\leqslant i<j\leqslant n$, then $i$ and $j$ are in the same block of
$K(\alpha)$ if and only if $\{k| i+1\leq k\leq j\}$ is the union of blocks of
$\alpha$.

This definition is the most common one, but not the best suited for our
purposes. We will often use instead another one that underlies the
equivalence between Eq.~(\ref{eq1}) and its restatement in terms of
Kreweras complements when $p=2$, see \cite[Exercise
14.3]{Nica-Speicher:Lectures}.
 
\begin{defi}\label{Kco2}
  Let $\alpha$ be a noncrossing partition in $\parti(n)$. The {\em Kreweras complement} $K(\alpha)$ is the
  unique noncrossing partition in $\parti(n)$ such that $(\alpha,K(\alpha)) $ is
  admissible and
  \begin{equation}\label{eq2}
  \alpha\circ K(\alpha)=1_{n}.
  \end{equation}
\end{defi}

\begin{proof} 
  For the definition to be consistent, one has to show the existence and
  uniqueness of $K(\alpha)$ solving Eq.~\ref{eq2}. The existence follows from
  the classical construction of the Kreweras complement
  \cite{Kreweras:1972}. Uniqueness follows from the strict monotonicity of the composition
  product $\circ$ (Lemma~\ref{monoto}).
\end{proof}

The same definition applies {\it mutatis mutandis} to define Kreweras complements
in $\parti(S)$. We denote by $K^S(\alpha)$ the Kreweras complement in $\parti(S)$ of
an element $\alpha$ in $\parti(S)$.

The Kreweras complement is a set automorphism (the equation $\alpha\circ
\beta=1_{n}$ can be solved uniquely for $\alpha$ given $\beta$), and a
non-involutive anti-automorphism of posets. Strict monotonicity indeed implies that it reverses the order:
$\alpha\divides \gamma\iff K(\gamma)\divides K(\alpha)$ and that, in this formula, $\alpha\neq \gamma\iff K(\gamma)\neq K(\alpha)$. See
\cite{Kreweras:1972} for the classical presentation and proofs.
Notice
that non-involutivity is equivalent to the noncommutativity of the composition product  \eqref{compoprod}:
%$\circ$:
$\alpha\circ\beta\neq\beta\circ\alpha$.

In particular, since $K(0_n)=1_n$ we have:
\begin{lemma}
  Given $\alpha\in\parti(n)$, the two intervals $[0_n,\alpha]$ and $[K(\alpha),1_n]$
  are anti-isomorphic lattices.
\end{lemma}

These notions generalize to the relative Kreweras complement setting:

\begin{defi}\label{Kco2rel} (Cf.~Nica--Speicher~\cite[Lecture 18]{Nica-Speicher:Lectures}; 
  the notion goes back to \cite{Nica-Speicher:9604011}.)
  Let $\alpha\divides \beta$ be two partitions in $\parti(n)$. The 
  {\em relative Kreweras complement}
  $K_\beta(\alpha)$ is the unique noncrossing partition in $\parti(n)$ such
  that $(\alpha,K_\beta(\alpha))$ is admissible and
  \begin{equation}\label{eq3}
  \alpha\circ K_\beta(\alpha)=\beta.
  \end{equation}
\end{defi}
 
The relative definition can be explained
% 
% actually boils down to the standard one: it can be
% understood 
as performing Kreweras complementation on each block of $\beta$.
This follows from Lemma~\ref{relatlatt}, and can be explained more directly
as follows. Given such a block $\beta_i$, one considers its sub-blocks in
$\alpha$. They form a noncrossing partition $\gamma_i$ of $\beta_i$. The
Kreweras complement of $\gamma_i$ in the set $\beta_i$ is a noncrossing
partition $K^{\beta_i}(\gamma_i)$. The element $K_\beta(\alpha)$ is then
obtained as the union of all the $K^{\beta_i}(\gamma_i)$.

This observation allows one to deduce the properties of the relative case from
the absolute case. In particular, $K_\beta$ is a set automorphism and an
anti-automorphism of posets of the interval $[0_n,\beta]$
(it reverses the order: $0_n\divides
\alpha\divides \nu\divides \beta \iff 0_n{=}K_\beta(\beta)\divides
K_\beta(\nu)\divides K_\beta(\alpha)\divides K_\beta(0_n){=}\beta$).
The next result follows immediately from this:
% As $K_\beta(0_n)=\beta$ and $K_\beta(\beta)=0_n$, we also get:

\begin{lemma}
  Given noncrossing partitions $\alpha, \beta\in\parti(n)$ with 
  $\alpha\divides \beta$, there are canonical 
  isomorphisms of lattices
  $$
  [0_n,\alpha]\op
  \stackrel{K_\beta}\isopil
  [K_\beta(\alpha),\beta]
  \qquad \text{ and } \qquad
  [\alpha,\beta]\op   \stackrel{K_\beta}\isopil
  [0_n,K_\beta(\alpha)]  .
  $$
  (Here $\op$ denotes the lattice with the opposite order.)
\end{lemma}

%%%%%%%%%%%%%%%

\subsection{Some applications}
\label{ssec:applications}

To finish this algebraic part of the article we show how the formalism
allows one to recover and rephrase two key results of the theory of noncrossing
partitions obtained respectively in \cite{Nica-Speicher:Lectures} and
\cite{Arizmendi-Vargas:2012}, with a view towards applications to free
probability.
The point is that the arithmetic formalism allows easily to perform
computations with Kreweras complements. For example, for
$(\alpha,\beta)$ admissible, we have
$$
\alpha\circ\beta\circ K(\alpha\circ\beta)=1_n=\alpha\circ K(\alpha),
$$
so that, for $(\alpha,\beta)$ admissible we always have
$$
K(\alpha)=\beta\circ K(\alpha\circ\beta).
$$

\begin{prop}\label{divassoc}
  Assume that $(\alpha,\beta,\gamma)$ is an admissible triple. Then 
  $$
  K_{\alpha\circ\beta\circ\gamma} (\alpha\circ\beta)
  = \gamma
  = K_{K_{\alpha\circ\beta\circ\gamma} (\alpha)} (K_{\alpha\circ\beta} (\alpha)).
  $$
\end{prop}

\begin{proof}
  The first equation is clear. For the second, use that 
  $K_{\alpha\circ\beta} (\alpha) =\beta$ and
  $K_{\alpha\circ\beta\circ\gamma}(\alpha)=\beta\circ\gamma$, so that
  $$
  K_{K_{\alpha\circ\beta\circ\gamma} (\alpha)}(K_{\alpha\circ\beta} (\alpha))
  = K_{K_{\alpha\circ\beta\circ\gamma} (\alpha)} (\beta)
  = K_{\beta\circ\gamma} (\beta)
  = \gamma.
  $$
\end{proof}

Recall from \cite{Arizmendi-Vargas:2012} that a $k$-preserving noncrossing
partition $\alpha$ in $\parti(kn)$ is called {\em $k$-completing} if and only if
$$
\alpha\vee\{\{1,\ldots,k\},\ldots,\{kn-n+1,\ldots,kn\}\}=1_{kn}.
$$
An admissible $k$-tuple in $\parti(n)$,
$(\alpha_1,\ldots,\alpha_k)$, is called {\em complete} if and only if
$\alpha_1\circ\cdots\circ\alpha_k=1_{n}$. 
Recall also that a multichain of length $k$
in a poset is a non-decreasing sequence of elements
$$
x_0\leq x_1\leq\cdots\leq x_k.
$$
See \cite{Mastnak-Nica:TAMS2010} for applications of multichains in the
lattice of noncrossing partitions to free probability.

The following proposition summarizes results due to Edelman~\cite{Edelman:1980}
and Arizmendi--Vargas~\cite{Arizmendi-Vargas:2012}.

\begin{prop}\label{canobij}
There are  canonical bijections between
\begin{enumerate}
\item admissible $k$-tuples in $\parti(n)$, 
\item $k$-preserving noncrossing partitions in $\parti(kn)$, 
\item multichains of length $k-1$ in the poset $\parti(n)$, 
\item complete admissible $(k{+}1)$-tuples in $\parti(n)$,
\item $(k{+}1)$-completing noncrossing partitions in $\parti((k+1)n)$.
\end{enumerate}
\end{prop}

\begin{proof}
\begin{itemize}
\item (1) $\Leftrightarrow$ (2): We already know that the two sets are in bijection by 
$$
	(\alpha_1,\ldots,\alpha_k)\longmapsto \alpha_1\shuffle_n\cdots\shuffle_n\alpha_k.
$$
\item (1) $\Leftrightarrow$ (3): The bijection is given by 
$$
	(\alpha_1,\ldots,\alpha_k)\longmapsto \alpha_1\divides 
\alpha_1{\circ}\alpha_2\divides \cdots\divides \alpha_1{\circ}\cdots{\circ}\alpha_k.
$$
The inverse bijection by 
$$
	\beta_1\divides \beta_2\divides \cdots\divides \beta_k
\longmapsto 
(\beta_1,K_{\beta_1}(\beta_2),\ldots,K_{\beta_{k-1}}(\beta_k)).
$$
\item (1) $\Leftrightarrow$ (4): The bijection is given by
$$
(\alpha_1,\ldots,\alpha_k)\longmapsto (\alpha_1,\ldots,\alpha_k,K(\alpha_1{\circ}\cdots{\circ}\alpha_k)).
$$
\item (4) $\Leftrightarrow$ (5): This follows from Lemma~\ref{translationlemma}.
\end{itemize}
\end{proof}

%%%%%%%%%%%%%%%
%%%%%%%%%%%%%%%

\section{From classical incidence (co)algebras to decalage}
\label{sec:class}

%%%%%%%%%%%%%%%

\subsection{The standard construction}
\label{ssec:standardcons}

To proceed further, it will be useful to recall classical 
M\"obius inversion \cite{Rota:Moebius}. 
Given functions $F,G:\N^\ast \to \C$ (called arithmetic functions),
the classical M\"obius inversion principle (which goes back to Euler) states
that 
$$
F(n) = \sum_{d \divides n} G(d) 
\qquad \text{if and only if} 
\qquad
G(n) = \sum_{d\divides n} F(d) \mu(n/d) ,
$$
where $\mu$ is the M\"obius function, given by
$$
\mu(n) = \begin{cases} (-1)^r & \text{ if  $n$ is a product of $r$ distinct 
primes} \\
0 & \text{ $n$ contains a square factor.}\end{cases}
$$

Since the work of Rota~\cite{Rota:Moebius} this is considered a
special case of M\"obius inversion for posets as follows.
Consider the poset of 
positive integers $\divposet$ with order given by divisibility, 
and its incidence coalgebra
$\incCoalg{\divposetnoparens}$ of intervals, with comultiplication given by
$$
	\Delta([n,m]):=\sum\limits_{n\divides k\divides m}[n,k]\otimes [k,m].
$$
The incidence algebra $\incAlg{\divposetnoparens}$ associated to $\divposet$ is the convolution 
algebra defined by duality from the coalgebra $\incCoalg{\divposetnoparens}$, in analogy with Definition~\ref{incidal}.
Its elements are linear functions on $\incCoalg{\divposetnoparens}$.
Among those functions are the zeta function $\zeta(n,m):=1$ for all $n\divides m$,
and its convolution inverse, the M\"obius function $\mu := \zeta^{\ast-1}$.
The M\"obius inversion formula in the poset version says that 
$$
	g(n,m)=\sum\limits_{n\divides k
	\divides m}f(n,k) 
	\quad \makebox{if and only if} \quad 
	f(n,m)=\sum\limits_{n\divides k \divides m}g(n,k) \mu(k,m).
$$
The link back to classical M\"obius inversion relies on the observation that 
both the zeta function and the M\"obius function have the 
property that their value on an interval $[n,m]$ depends only
on the number ${m/n}$; the functions with this property form a subalgebra
$\redIncAlg{\divposetnoparens}\subset \incAlg{\divposetnoparens}$ 
called the reduced incidence algebra. That these functions form indeed a 
subalgebra follows from the observation that the two intervals $[n,m]$ and 
$[1,\frac{m}{n}]$ are canonically isomorphic, so as to justify the change of 
summation in the middle step of the calculation
$$\textstyle
(f\ast g)(n,m)
=
\sum\limits_{n\divides k\divides m}f(n,k)g(k,m)
=
\sum\limits_{n\divides k\divides m}f(\frac{n}{n}, \frac{k}{n})g(\frac{k}{n}, \frac{m}{n})
= 
\sum\limits_{1\divides d\divides \frac{m}{n}}f(1, d)g(d, \frac{m}{n})
=
(f\ast g)(1, \frac{m}{n}).
$$
These functions can also be described as those of the form $f(n,m) = 
F(m/n)$ for some function $F$ on the multiplicative monoid
$\multmonoid$. This gives a canonical 
identification 
$$
\redIncAlg{\divposetnoparens}
\simeq \incAlg{\multmonoidnoparens} .
$$
The latter is the convolution algebra of the incidence coalgebra of 
$\multmonoid$,
with convolution product $\circ$ given by
$$
(F\circ G)(m) = \sum\limits_{i \cdot j = m} F(i)G(j),
$$
which is the convolution product associated to the comultiplication
$$
\Delta(m):=\sum\limits_{i \cdot j = m} i \otimes j.
$$
In other words, classical M\"obius inversion, although generally 
formulated in the incidence algebra of the divisibility poset, 
is actually rather a property of the multiplicative monoid, the two aspects 
being related via
 the homomorphism of coalgebras
\begin{eqnarray*}
  \incCoalg{\divposetnoparens} & \longrightarrow & \incCoalg{\multmonoidnoparens}  \\
  {}[n,m] & \longmapsto & m/n .
\end{eqnarray*}

%%%%%%%%%%%%%%%

\subsection{Categorical and simplicial interpretation}
\label{ssec:categoricalinterp}

The relationship between the two approaches (intervals in posets vs.~elements in a
monoid) is formulated elegantly by regarding both posets and monoids as examples
of categories: Recall that a poset can be regarded as a category where there is an
arrow $x\to y$ whenever $x\leq y$, and that a monoid $M$ gives rise to a category
with a single object and whose arrows are the elements of $M$, the composition of
arrows being given by monoid multiplication in $M$. It was observed by Content,
Lemay and Leroux~\cite{Content-Lemay-Leroux} that the assignment $a\divides b
\mapsto \frac{b}{a}$ constitutes a functor from the category $\divposet$ to the
category $\multmonoid$. Furthermore, this functor is CULF (``conservative'' and
having ``unique lifting of factorizations''), which they identified as the class
of functors that induce coalgebra homomorphisms covariantly at the level of
incidence coalgebras, or equivalently, algebra homomorphisms contravariantly
between the incidence algebras. The functor thus induces the above coalgebra
homomorphism
$$
\incCoalg{\divposetnoparens} \onto \incCoalg{\multmonoidnoparens}
$$
(which is a quotient map) and dually the embedding of convolution
algebras
$$
\incAlg{\multmonoidnoparens} \into \incAlg{\divposetnoparens} .
$$
  
More recently, it was observed that in the setting of simplicial
sets this relationship is an instance of a very general phenomenon: the
nerve of the divisibility poset $\divposet$ is the lower decalage of the
bar complex of the monoid $\multmonoid$, that the functor is the
canonical map that always exists from a decalage back to the original
simplicial set, and that this functor is
CULF whenever the simplicial set is the nerve of a category (or more generally 
a decomposition space)~\cite{Galvez-Kock-Tonks:1612.09225}. In this way, one may say loosely
that $\divposet$ is just a ``shift'' of $\multmonoid$. Let us briefly
explain the decalage viewpoint, as we will find exactly the same situation
for noncrossing partitions.

One way to have posets, monoids and categories on equal footing is in terms
of their nerves, which are simplicial sets, 
i.e.~functors $X:\simplexcategory\op\to\kat{Set}$
where $\simplexcategory$ is the category of
finite nonempty ordinals $\{0,..,n\}$ and order-preserving maps.
The unique order-preserving injection from
$\{0,1,\dots,n-1\}$ into $\{0,1,\dots,n\}$ whose image does not contain $i$
induces a {\em face} map $d_i : X_n \to X_{n-1}$. The unique
order-preserving surjection from $\{0,1,\dots,n\}$ to $\{0,1,\dots,n-1\}$
that maps $i$ and $i+1$ to $i$ induces a {\em degeneracy} map $s_i :
X_{n-1} \to X_n$. The relations obeyed by the face and degeneracy maps
are called the simplicial identities; they can be used to define simplicial
sets without using the language of categories and functors.
Recall
that the nerve of a category $\CC$ is the simplicial set $X:=N\CC :
\simplexcategory\op\to\kat{Set}$ where a $k$-simplex is a string of $k$
composable arrows in $\CC$. In particular, $X_0$ is the set of objects, and
$X_1$ is the set of arrows.
In the special case of a poset, $X$ is also
called the {\em order complex\ }\!: $X_0$ is the set of elements in the poset,
$X_1$ is the set of all intervals, and $X_k$ is the set of all multichains
of length $k$ (meaning that there are $k$ steps, or equivalently $k+1$
poset elements in the multichain). For a monoid $M$, the nerve is also
called the {\em bar complex},\footnote{The word ``bar'' comes from the first
paper on the subject (Eilenberg--Mac Lane), where an element in $X_k$ was
denoted $x_1|x_2|\cdots|x_k$. Here we cannot use that notation, as
we employ the symbol $\divides $ for divisibility and ordering.} and is
traditionally denoted $BM$. Here the set of $k$-simplices $X_k$ is the set
$M^k$. The outer face maps project away the first or last element
of a $k$-tuple, 
while the inner face maps multiply adjacent elements.
  
For any simplicial set $X$, the {\em lower decalage} of $X$, denoted
$\Decbot(X)$, is the simplicial set obtained by forgetting $X_0$ and
shifting down all higher $X_k$, so that
$$
\Decbot(X)_k = X_{k+1} .
$$
This is a simplicial set again: the face and degeneracy maps are all the
face and degeneracy maps of $X$ except $d_0$ and $s_0$, and they are
shifted down by one index, so that the new $d_i$ are the old $d_{i+1}$ (and
the new $s_i$ are the old $s_{i+1}$). There is a canonical simplicial map
$\Decbot(X) \to X$, often called the {\em dec map}, given by using the
original $d_0$ maps. Altogether we get (degeneracy maps are not represented):
\[
\begin{tikzcd}[column sep={22mm,between origins},row sep={15mm,between origins}]
X: & X_0  &
  \ar[l, shift left=3.5pt,  "d_0" on top, shorten >=1mm,shorten <=1mm]
  \ar[l, shift right=3.5pt,  "d_1" on top, shorten >=1mm,shorten <=1mm]
X_1 & 
  \ar[l, shift left=7pt, "d_0" on top, shorten >=1mm,shorten <=1mm]
  \ar[l, shift right=7pt, "d_2" on top, shorten >=1mm,shorten <=1mm]
  \ar[l, "d_1" on top, shorten >=1mm,shorten <=1mm]
X_2 & \cdots  
\\
\Decbot(X): & 
X_1 \ar[u, "d_0"] &
  \ar[l, shift left=3.5pt,  "d_1" on top, shorten >=1mm,shorten <=1mm]
  \ar[l, shift right=3.5pt,  "d_2" on top, shorten >=1mm,shorten <=1mm]
X_2 \ar[u, "d_0"] & 
  \ar[l, shift left=7pt, "d_1" on top, shorten >=1mm,shorten <=1mm]
  \ar[l, shift right=7pt, "d_3" on top, shorten >=1mm,shorten <=1mm]
  \ar[l, "d_2" on top, shorten >=1mm,shorten <=1mm]
X_3 \ar[u, "d_0"] & \cdots
\end{tikzcd}
\]
The simplicial identities ensure that this map is a  map of simplicial sets.
  
Applying this construction to the bar complex of $\multmonoid$, we obtain
  \[
  \begin{tikzcd}[column sep={30mm,between origins},row sep={20mm,between origins}]
  B\N^\ast: & *  & 
  \ar[l, shift left=4pt,  "d_0" on top, shorten >=2mm,shorten <=2mm]
  \ar[l, shift right=4pt,  "d_1" on top, shorten >=2mm,shorten <=2mm]
\N^\ast &
  \ar[l, shift left=8pt, "d_0" on top, shorten >=2mm,shorten <=2mm]
  \ar[l, shift right=8pt, "d_2" on top, shorten >=2mm,shorten <=2mm]
  \ar[l, "d_1" on top, shorten >=2mm,shorten <=2mm]
\N^\ast{\!\times}\N^\ast & \cdots  
\\
  \Decbot(B\N^\ast): & \N^\ast 
  \ar[u, "d_0", shorten >=1mm,shorten <=1mm] &
  \ar[l, shift left=4pt, pos=0.45, "d_1" on top, shorten >=1mm,shorten <=1mm]
  \ar[l, shift right=4pt, pos=0.45, "d_2" on top, shorten >=1mm,shorten <=1mm]
\N^\ast{\!\times}\N^\ast 
\ar[u, "d_0", shorten >=2mm,shorten <=2mm] & 
  \ar[l, shift left=8pt, "d_1" on top, shorten >=1mm,shorten <=1mm]
  \ar[l, shift right=8pt,  "d_3" on top, shorten >=1mm,shorten <=1mm]
  \ar[l,  "d_2" on top, shorten >=1mm,shorten <=1mm]
\N^\ast{\!\times}\N^\ast{\!\times}\N^\ast
  \ar[u, "d_0", shorten >=2mm,shorten <=2mm] & \cdots
  \end{tikzcd}
  \]
  In the left part of the bottom row, the face maps $d_1$ and $d_2$ send a pair $(a,b)$ to 
  $ab$ and $a$, respectively. Clearly we have $a \divides  ab$, which can be 
  interpreted as an interval in the divisibility poset. This is in fact 
  part of a canonical isomorphism of simplicial sets between the lower decalage of the bar complex over the monoid of the integers and the nerve of the divisibility poset:
  $$
  \Decbot(B\N^\ast) \isopil N \divposet .
  $$
  In degree $k-1$ this isomorphism is given as the following bijection between $k$-tuples and
  $(k-1)$-multichains:
  $$
  (a_1,a_2,\dots,a_k) \mapsto a_1 \divides  a_1 a_2 \divides  a_1 a_2 a_3 \divides \dots\divides  a_1 a_2 a_3 \dots a_k .
  $$
  The key point is that the face and degeneracy maps match up too,
  assembling the bijections into an isomorphism of simplicial sets. As an
  example, the bottom face map $d^{\Decbot(B\N^\ast)}_0$ is the original
  $d^{B\N^\ast}_1$ so its effect on $(a_1 ,a_2,a_3,a_4)\in (\N^\ast)^4$ is
  to multiply $a_1$ and $a_2$, giving $(a_1 a_2,a_3,a_4)$, which maps to
  $a_1 a_2 \divides a_1 a_2 a_3 \divides a_1 a_2 a_3 a_4$ by the
  isomorphism. We obtain the same value when the face map of the nerve of
  the divisibility poset $d^{N \divposet}_0$ acts on $a_1 \divides a_1 a_2
  \divides a_1 a_2 a_3 \divides a_1 a_2 a_3 a_4$.

%%%%%%%%%%%%%%%
%%%%%%%%%%%%%%%

\section{Coalgebraic and topological structures}
\label{sec:coalg}

The construction of incidence coalgebra, incidence algebra, and M\"obius
inversion makes sense for simplicial sets more general than nerves of
categories. The natural level of generality is that of {\em decomposition
spaces}~\cite{Galvez-Kock-Tonks:1512.07573} (also called $2$-Segal
spaces~\cite{Dyckerhoff-Kapranov:1212.3563}). They are simplicial sets more
general than nerves of categories, and in particular they include the classical 
cases of posets (Rota~\cite{Rota:Moebius})
and monoids (Cartier--Foata~\cite{Cartier-Foata}); 
see \cite{Galvez-Kock-Tonks:1612.09225} for many examples beyond posets and 
monoids. 
It is a general fact that the lower decalage of a
decomposition space is always the nerve of a category, and that the dec
map is always CULF~\cite{Galvez-Kock-Tonks:1512.07573}.
  
Many combinatorial coalgebras can be shown not to be the incidence
coalgebra of any category (or poset or monoid), but virtually all of them
can be realized by decomposition spaces (according to
\cite{Galvez-Kock-Tonks:1612.09225}). Where categories encode the ability
to compose, decomposition spaces owe their name to having instead the
ability to decompose, as happens abundantly for combinatorial structures,
even in situations where one cannot always compose. As a special case,
partially defined composition laws and multivalued composition laws can
often fruitfully be interpreted as defining decomposition spaces.
This happens in particular for partial monoids, as first observed by
Bergner et
al.~\cite{Bergner-Osorno-Ozornova-Rovelli-Scheimbauer:1609.02853}: the bar
complex of a partial monoid is a decomposition space. Its $k$-simplices are
given by {\em admissible} $k$-tuples of elements in the partial
monoid~\cite{Segal:1973}. The partial associativity condition satisfied by
partial monoids translates precisely into the axioms for a decomposition
space in this case. In particular, by the general theory of decomposition
spaces, partial monoids have incidence (co)algebras (where the
(co)multiplication becomes an everywhere-defined operation). The 
comultiplication is exactly the same as for genuine monoids: 
\begin{equation}\label{Delta{m}}
\Delta(m) = \sum_{m_1\cdot m_2=m} m_1 \otimes m_2.
\end{equation}
   
In this section we treat accordingly the partial monoid of noncrossing
partitions. In a first step (Subsection~\ref{ss1}), we perform the
constructions ``classically'', appealing only to standard Rota-type algebra
arguments. In a second step (Subsection~\ref{ss3}), just as we did for the
positive integers, we show how these results can be interpreted
categorically and simplicially.

%%%%%%%%%%%%%%%
    
\subsection{Noncrossing partitions coalgebras}
\label{ss1}

Before coming to the partial monoid of noncrossing partitions, we look at the the
poset of noncrossing partitions, and note the following analogy with the reduced
incidence algebra of the divisibility poset:

\begin{prop}\label{subsetDiv}
  The subspace $\redIncAlg{\parti(n),\divides\,}\subset \incAlg{\parti(n),\divides\,}$
  of functions $f$ whose  value $f(\alpha,\beta)$ depends only on the relative 
  Kreweras complement $K_\beta(\alpha)$ is a subalgebra, called 
  the {\em reduced incidence algebra}.
\end{prop}
\noindent
(This result can hardly be considered new; versions of it go back to 
Speicher~\cite{Speicher:multiplicative}.)
The key ingredient in the proof is the following lemma, which is a 
variation of Proposition~\ref{divassoc}:
\begin{lemma}[{Nica--Speicher~\cite[Lemma 
  18.9]{Nica-Speicher:Lectures}}]\label{Krel}
  Given noncrossing partitions $\alpha \divides \beta \divides \gamma$, we have
  \begin{enumerate}
	\item $K_{\beta}(\alpha) \divides K_\gamma(\alpha)$
	\item There are canonical isomorphisms of intervals
	$
  [K_{\beta}(\alpha), K_\gamma(\alpha)] \simeq [0_n , K_\gamma(\beta)] \simeq 
  [\beta,\gamma].
  $
\item $K_{K_\gamma(\alpha)}(K_\beta(\alpha)) = K_\gamma(\beta)$.
  \end{enumerate}
\end{lemma}
\begin{cor}\label{cor:K_beta(alpha)}
  Given noncrossing partitions $\alpha \divides \gamma$,
  there is a canonical isomorphism of intervals
  \begin{eqnarray*}
    {}[\alpha,\gamma] & \isopil & [0_n, K_\gamma(\alpha)]  \\
    \beta & \mapsto & \sigma:=K_\beta(\alpha)
  \end{eqnarray*}
\end{cor}
\begin{proof}[Proof of Proposition~\ref{subsetDiv}]
	Suppose $f$ and $g$ are functions that only depend on the Kreweras complement.
	This assumption is used in the second equality below, together with
	Lemma~\ref{Krel}:
    $$(f\ast g)(\alpha,\gamma) 
	= \sum_{\alpha\divides\beta\divides\gamma} f(\alpha,\beta)\,g(\beta,\gamma)
	= \sum_{\alpha\divides\beta\divides\gamma} f(0_n,K_\beta(\alpha))\,g(K_\beta(\alpha),K_\gamma(\alpha))
	$$ $$= \sum_{0_n\divides\sigma\divides K_\gamma(\alpha)} f(0_n,\sigma)\,g(\sigma,K_\gamma(\alpha))
	= (f\ast g)(0_n,K_\gamma(\alpha)) .
	$$
	The change of summation in the third step of the calculation is justified by 
	Corollary~\ref{cor:K_beta(alpha)}.
  \end{proof}
  In analogy with the case of the divisibility poset, the functions here can 
  also be characterized as those with $f(\alpha,\beta) = F( K_\beta(\alpha))$
  for some function on $\parti(n)$.  These functions in turn form the incidence 
  coalgebra of the partial monoid, which can be considered as a quotient of the 
  raw incidence coalgebra of the noncrossing partitions lattice, as we now proceed to explain.

Any (locally finite) partial monoid gives rise to a coalgebra by the following
process which is the same as for monoids and relies on associativity, unitality
and the fact that $(\alpha\circ\beta)\circ\gamma$ is defined if and only if
$\alpha\circ(\beta\circ\gamma)$ is. The proofs of coassociativity and counitality
(as well as M\"obius inversion, in many cases) are also the same, or one can
invoke the more general results for decomposition
spaces~\cite{Galvez-Kock-Tonks:1512.07573}.

\begin{defi}
  The incidence coalgebra $\incCoalg{\parti(n), \circ}$ is spanned as a 
  vector space by $\parti(n)$, and has comultiplication
  induced by the partial monoid structure of
  $(\parti(n),\circ\,)$:
  $$
  \Delta_{\circ}(\pi) = \sum\limits_{\alpha\circ\beta=\pi}\alpha\otimes \beta 
  = \sum\limits_{\alpha \divides \pi}\alpha\otimes K_\pi(\alpha),
  $$
  with counit $\partial_{\circ}(\pi):=1$ if $\pi=0_n$ and zero otherwise.
\end{defi}

We have finally a compatibility property, similar to the one established in
the framework of classical M\"obius inversion. It will follow from
general theoretical arguments of Subsection~\ref{ss3} below, but we also provide
here a direct proof.

\begin{prop}\label{homcaolnc}
  The map
  \begin{eqnarray*}
    \Psi:\incCoalg{\parti(n),\divides\;} & \longrightarrow & 
	\incCoalg{\parti(n),\circ}  \\
    {}[\alpha,\beta] & \longmapsto & K_\beta(\alpha)
  \end{eqnarray*}
  is a homomorphism of coalgebras.
\end{prop}

\begin{proof}
 Given $\alpha \divides \gamma$, we check that $\Psi$ preserves the 
  comultiplication:
  $$
  (\Psi\otimes\Psi)( \Delta([\alpha,\gamma]))
  =\sum\limits_{\alpha\divides \beta \divides \gamma}\Psi([\alpha,\beta])\otimes\Psi([\beta,\gamma ])
  =\sum\limits_{\alpha\divides \beta \divides \gamma} K_\beta(\alpha) \otimes K_\gamma(\beta)
  =\sum\limits_{\alpha\divides \beta \divides \gamma} K_\beta(\alpha) \otimes 
  K_{K_\gamma(\alpha)}(K_\beta(\alpha)),
$$
where the last step used Lemma~\ref{Krel} (item 3).
Now we use the isomorphism $[\alpha,\gamma] \simeq [0_n, K_\gamma(\alpha)]$, 
$\beta \mapsto \sigma := K_\beta(\alpha)$ of Corollary~\ref{cor:K_beta(alpha)} to 
continue the calculation:
$$
  =\sum\limits_{0_n\divides \sigma \divides K_\gamma(\alpha)} \sigma \otimes 
  K_{K_\gamma(\alpha)}(\sigma) 
= \sum_{\sigma\circ \pi = K_\gamma(\alpha)} \sigma \otimes \pi 
= \Delta_\circ( K_\gamma(\alpha))
= \Delta_\circ( \Psi([\alpha,\gamma])
$$
which establishes the comultiplicativity of $\Psi$. (The fact that $\Psi$ preserves the counit is obvious.)

\end{proof}

\begin{remark}
 Proposition~\ref{subsetDiv} and Lemma~\ref{homcaolnc} together identify
  the reduced incidence coalgebra $\redIncCoalg{\parti(n),\divides\,}$ of the
  poset with the incidence coalgebra $\incCoalg{\parti(n),\circ}$ of the partial
  monoid. The reduction taken here --- identifying intervals in $\parti(n)$ if
  they have the same relative Kreweras complements --- thus matches the partial
  monoid. It should be noted that there are other possibilities for reduction,
  that is, other natural quotients to consider. One quotient construction
  consists in identifying intervals if
  they have the same {\em fibre monomial}, namely for $\alpha\divides \beta$ the
  same family of preimages of the blocks in $\beta$. This is the reduction used in
  our previous paper \cite{EbrahimiFard-Foissy-Kock-Patras:1907.01190} in
  connection with the block-substitution operad, and later studied further by
  Celestino et al.~\cite{Celestino-EbrahimiFard-Nica-Perales-Witzman}. To set up
  this reduction, the natural thing is to define a coalgebra map to the 
  polynomial algebra on noncrossing partitions (with comultiplication induced by 
  a certain operad structure), sending an interval to
  its fibre monomial. 
  The two reductions are not comparable: neither factors through the other.
  The relative Kreweras complement does not determine the fibre monomial and the 
  fibre monoial does not determine the Kreweras complement.

  Another possible notion of reduction, 
  which goes further than both of the previous two options, is to 
  identify intervals if their Kreweras complements have the same type 
  (sizes of blocks).
  (This reduction can be factored either through the
  Kreweras complement or through the fibre monomial.)
  This is the reduction used by 
  Speicher~\cite{Speicher:multiplicative}, although formulated differently.
  In particular, Speicher's families of multiplicative functions,
  a particular class of families of linear forms defined
  simultaneously on all the $\incCoalg{{\parti(n),\divides\,}},\ n\in\N$,
  have the property of
  only depending on the relative Kreweras complement, and can therefore
  be considered as families of functions on all the $\incCoalg{\parti(n),\circ}$, 
  $n\in \N$.

\end{remark}

%%%%%%%%%%%%%%%
 
\subsection{Categorical and simplicial aspects}
\label{ss3}

In this subsection we show that the partial monoid $(\parti(n),\circ)$ relates
to the noncrossing partitions lattice $(\parti(n),\divides\;)$ precisely as
the multiplicative monoid $\multmonoid$ relates to the divisibility poset
$\divposet$.

\begin{prop}\label{Dec-iso}
  The lower decalage of the (bar complex of the) partial monoid 
  $(\parti(n), \circ)$ is isomorphic to the (nerve of the) poset of noncrossing partitions
  $(\parti(n),\divides\;)$.
\end{prop}

\begin{proof}
  As explained, the bar complex $X$ of $(\parti(n), \circ)$ has 
  $X_0=*$ (singleton) and $X_1$ the set of noncrossing 
  partitions. The set $X_2$ is the set of admissible pairs of 
  noncrossing partitions. More generally $X_k$ is the set 
  of admissible $k$-tuples of noncrossing partitions.
  A $k$-simplex of the lower decalage is thus
  an admissible $(k{+}1)$-tuple, and by Proposition~\ref{canobij} this defines
  uniquely a $k$-multichain in the noncrossing partitions lattice, i.e.~a 
  $k$-simplex in the nerve of $(\parti(n),\divides\;)$. So in each 
  simplicial degree we have the required bijection.

  The more interesting part is to check also that the face and degeneracy maps match
  up. This check is completely analogous to the case of 
  the divisibility poset and the multiplicative monoid of positive integers.
  As a sample, let us consider a $2$-simplex in the lower decalage 
  of $X$ (so an admissible $3$-tuple)
  $$
  (\alpha_1,\alpha_2,\alpha_3) .
  $$
  The three faces (applying $d_0, d_1, d_2$ of the decalage, which are $d_1, d_2, d_3$ of the bar complex) are, respectively
  $$
  (\alpha_1\circ\alpha_2,\alpha_3), \quad   
  (\alpha_1,\alpha_2\circ\alpha_3), \quad
  (\alpha_1,\alpha_2) ,
  $$
  and their images in the nerve of the noncrossing partitions poset
  under the bijections are
  the intervals
  $$
    \alpha_1\circ\alpha_2 \divides \alpha_1\circ\alpha_2\circ\alpha_3, \quad   
	\alpha_1 \divides \alpha_1\circ\alpha_2\circ\alpha_3, \quad
	\alpha_1 \divides \alpha_1\circ\alpha_2 .
  $$  
  On the other hand, the $3$-tuple $(\alpha_1,\alpha_2,\alpha_3)$ is sent 
  to the $2$-multichain
  $$
  \alpha_1 \divides \alpha_1 \circ \alpha_2 \divides \alpha_1 \circ 
  \alpha_2 \circ \alpha_3
  $$
  whose $3$ faces (applying $d_0,d_1, d_2$ of the poset's nerve) are the same three 
  intervals.
\end{proof}

\begin{remark}
  It is quite rare for a poset (or category) to admit an ``undecking'' like
  this --- a simplicial set whose decalage is the given category. According to
  Garner--Kock--Weber~\cite{Garner-Kock-Weber:1812.01750} the existence of
  an undecking amounts to the category having the structure of unary
  operadic category in the sense of Batanin and
  Markl~\cite{Batanin-Markl:1404.3886}, a general abstract framework for
  operad-like structures. In particular, the noncrossing partitions lattice
  is thus an example of an unary operadic category, where the so-called
  fibre functor is given by the Kreweras complement. As far as we know,
  this example of operadic category had not been observed before --- it is
  of a rather different flavour than the usual examples of operadic
  categories. (The undecking relevant to operadic categories is actually
  {\em upper} decalage, not lower, but since the noncrossing partitions
  lattice is self-dual, in the present situation this detail is not
  important.)
\end{remark}

Composing the simplicial isomorphism of Proposition~\ref{Dec-iso}
with the dec map we get a CULF functor from the nerve of the 
noncrossing partitions lattice to the bar complex of the partial monoid.
In simplicial degree $1$, this map sends an interval in the noncrossing
partitions lattice to its relative Kreweras complement:
\begin{equation}\label{Kab}
  \alpha\divides \beta   \quad  \mapsto \quad K_\beta(\alpha) .
\end{equation}
This thus defines a coalgebra homomorphism 
$$
\incCoalg{\parti(n),\divides\;} \onto \incCoalg{\parti(n),\circ} 
$$
which
coincides with the one in Proposition~\ref{homcaolnc}, with the same
description as in \eqref{Kab}, and, dually, the algebra homomorphism
$$
\incAlg{\parti(n),\circ} \into \incAlg{\parti(n),\divides\;}
$$
on the
dual incidence algebras (the inclusion of
Proposition~\ref{subsetDiv} of those functions whose values on
an interval only depends on its relative Kreweras complement).

\small
\vspace{1cm}

\noindent {\bf Acknowledgments.} 
L.F. and F.P. acknowledge support from the grant ANR-20-CE40-0007
Combinatoire Alg\'ebrique, R\'esurgence, Probabilit\'es Libres et Op\'erades.
J.K. was supported by grant
PID2020-116481GB-I00 (AEI/FEDER, UE) of Spain and grant
10.46540/3103-00099B from the Independent Research Fund Denmark, and was
also supported through the Severo Ochoa and Mar\'ia de Maeztu Program for
Centers and Units of Excellence in R\&D grant number CEX2020-001084-M. 
K.E.F. was supported by the Research Council of Norway through project 302831 {\it{Computational Dynamics and Stochastics
on Manifolds}} (CODYSMA). K.E.F. and F.P. would also like to thank the Centre for Advanced Study (CAS) in Oslo for its warm hospitality and financial support during the research project {\it{Signatures for Images}} (SFI).
F.P. was also supported by the ANR -- FWF project PAGCAP.

% \bibliographystyle{scplain}
% \bibliography{mauri}

\addcontentsline{toc}{section}{References}

\end{document}